\documentclass[11pt,a4paper]{amsart}
\usepackage{amsmath,amssymb,epsfig}

\usepackage{algorithm}
\usepackage{algpseudocode}
\usepackage{algorithmicx}

\usepackage{url}
\usepackage{array}

\usepackage{color}
\usepackage{gauss}

\usepackage{graphicx}
\usepackage{epsfig}
\usepackage{subfig}
\usepackage{grffile}

\usepackage{tikz}
\usepackage[vlined,lines numbered,ruled,algo2e]{algorithm2e}

\usepackage{amsthm}
\newtheorem{theorem}{Theorem}[section]
\newtheorem{lemma}[theorem]{Lemma}

\newtheorem{corollary}[theorem]{Corollary}

\newcommand{\C}{{\mathbb C}}

\newcommand{\calK}{{\mathcal K}}

\usepackage{amsaddr}

\title[A bivariate extension of the Crouzeix-Palencia result]{A bivariate extension of the Crouzeix-Palencia result with an application to Fr\'echet derivatives of matrix functions}

\author[M. Crouzeix \and D. Kressner]{Michel Crouzeix}

\address{Universit\'e de Rennes 1, CNRS, IRMAR - UMR 6625, F-35000 Rennes, France}

\author[]{Daniel Kressner}

\address{\'Ecole polytechnique f\'ed\'erale de Lausanne, Institute of Mathematics, CH-1015 Lausanne, Switzerland. E-mail: daniel.kressner@epfl.ch.}

\date{\today}



\begin{document}

\maketitle

\begin{abstract}
A result by Crouzeix and Palencia states that the spectral norm of a matrix function $f(A)$ is bounded by $K = 1+\sqrt{2}$ times the maximum of $f$ on $W(A)$, the numerical range of $A$. The purpose of this work is to point out that this result extends to a certain notion of bivariate matrix functions; the spectral norm of $f\{A,B\}$ 
is bounded by $K^2$ times the maximum of $f$ on $W(A)\times W(B)$. As a special case, it follows that the spectral norm of the Fr\'echet derivative of $f(A)$ is bounded by 
$K^2$ times the maximum of $f^\prime$ on $W(A)$. An application to the convergence analysis of certain Krylov subspace methods and the extension to functions in more than two variables are discussed. 
\end{abstract}

\section{Introduction}

The numerical range of a matrix $A\in \C^{n \times n}$ is the set 
\[
 W(A) := \{ v^* A v:\, v \in \C^n, \|v\|_2 = 1\},
\]
where we let $\|\cdot\|_2$ denote the Euclidean norm of a vector and the spectral norm of a matrix.
Consider the matrix function $f(A)$ for an analytic function $f: \Omega \subset \C \to \C$ with $W(A) \subset \Omega$.
A result by Crouzeix and Palencia~\cite{CrouzeixPalencia2017} shows that
\begin{equation} \label{eq:scalarcrouzeixpalencia}
 \|f(A)\|_2 \le (1+\sqrt{2}) \max_{z \in W(A)} |f(z)|.
\end{equation}
The purpose of this work is to point out that an analogous result holds for a certain notion of multivariate matrix functions.

To provide some intuition on the multivariate matrix functions considered in this work, let us first consider the bivariate polynomial $p(x,y) = 1 + xy + x^3 y^2$. For matrices $A \in \C^{m\times m}$ and 
$B \in \C^{n\times n}$, evaluating $p$ in the commuting matrices 
$I\otimes A$, $B \otimes I$ gives
\[
 p\{A,B\} := I + B \otimes A + B^2 \otimes A^3 \in \C^{mn\times mn},
\]
where $\otimes$ denotes the usual Kronecker product.
Equivalently, $p\{A,B\}$ can be viewed as the linear operator on $\C^{m \times n}$ defined
by $p\{A,B\}: X\mapsto X + A X B^T + A^3 X (B^T)^2$, where $B^T$ is the complex transpose of $B$. For a general bivariate function $f$ analytic on a domain $\Omega \subset \C \times \C$ containing the Cartesian product of the eigenvalues of $A$ and $B$, we define
\begin{equation} \label{def:bivariatefunction}
  f\{A,B\} := -\frac{1}{4\pi^2} \oint_{\Gamma_A} \oint_{\Gamma_B} f(x,y) (yI-B)^{-1} \otimes (xI-A)^{-1}\,\text{d}y\,\text{d}x,
\end{equation}
for closed contours $\Gamma_A$ and $\Gamma_B$ enclosing the eigenvalues of $A$ and $B$, respectively, and satisfying $\Gamma_A \times \Gamma_B \subset \Omega$. As explained in~\cite{Kressner2014bivariate}, this definition represents a special case of the well established notion of evaluating a multivariate holomorphic function in elements from a commutative Banach algebra; see, e.g.,~\cite{Biller2007} for an introduction. The definition~\eqref{def:bivariatefunction} is also closely related to the notion of double operator integrals~\cite{Skripka2020}.

Assuming additionally that $f$ is analytic in a domain containing  $W(A)\times W(B)$, our main result (see Theorem~\ref{theorem:main}) states that
\begin{equation} \label{eq:newresult}
  \|f\{A,B\}\|_2 \le (1+\sqrt{2})^2\, \|f\|_{W(A)\times W(B)},
\end{equation}
where $\|f\|_{W(A)\times W(B)}$ denotes the maximum of $|f|$ on $W(A) \times W(B)$.
The constant $(1+\sqrt{2})^2$ in~\eqref{eq:newresult} is worse than the one in~\eqref{eq:scalarcrouzeixpalencia} but it is in fact not possible to reduce the constant in~\eqref{eq:newresult} to $1+\sqrt{2}$. To see this, let $A = B = \begin{bmatrix} 0 & 1 \\ 0 & 0 \end{bmatrix}$ and $f(x,y) := xy$. Because the numerical range of $A$ is a disc centered at zero with radius $1/2$, we have
$\|f\|_{W(A)\times W(B)} = 1/4$.
On the other hand, $\|f\{A,B\}\|_2 = \|B\otimes A\|_2 = 1$, which shows that the constant in~\eqref{eq:newresult} must be at least $4$. This is entirely analogous to the univariate bound~\eqref{eq:scalarcrouzeixpalencia} for which the constant is known to be at least $2$ and, in fact, conjectured to be exactly $2$; see~\cite{Cadwell2018,Glader2018,Greenbaum2018,Ransford2018} for recent progress in this direction.

In Section~\ref{sec:multivariate}, we will see that the result~\eqref{eq:newresult} easily extends to functions in more than two variables; with a constant that (necessarily) grows exponentially with the number of variables.

Our result~\eqref{eq:newresult} has important implications in a variety of applications, including norm estimates for derivatives of matrix functions and the convergence analysis of certain Krylov subspace methods; see Sections~\ref{sec:derivatives} and~\ref{sec:krylov}, respectively.

\subsubsection*{Related work}
When $f(x,y) = g(x+y)$ for some univariate function $g$, a result by Starke~\cite[Corollary 3.2]{Starke1993} combined with~\eqref{eq:scalarcrouzeixpalencia} implies~\eqref{eq:newresult}, in fact with the lower constant $1+\sqrt{2}$; see also~\cite[Remark 1]{Kressner2019bivariate}. Existing results for general functions include work by Gil', such as~\cite[Theorem 1.1]{Gil2011}, and~\cite[Lemma 3]{Kressner2019bivariate}. These bounds are significantly more complicated than~\eqref{eq:newresult} and depend on additional quantities, such as the distance to normality of $A$ and/or $B$.

It is simple to see that~\eqref{eq:newresult} holds with constant $1$ when $A,B$ are both normal. This becomes more subtle when replacing the spectral norm by other Schatten norms, a question that has been studied extensively in the literature on double and multiple operator integrals~\cite{Skripka2019}.

The von Neumann inequality is a variant of~\eqref{eq:scalarcrouzeixpalencia}, which states that $\|f(A)\|_2 \le \|f\|_{\mathbb D}$ for $\|A\|_2 \le 1$, assuming that $f$ is analytic on the open unit disk $\mathbb D$ and continuous on $\overline{\mathbb D}$. This result has been extended in various ways to multivariate functions; see~\cite[Sec. 37.4]{BadeaBeckermann2013} for a survey. In particular, applying the seminal result by Ando~\cite{Ando1963} to the commuting matrices 
$I\otimes A$, $B \otimes I$ one obtains
\begin{equation} \label{eq:ando}
 \|f\{A,B\}\|_2 \le \|f\|_{\mathbb D \times \mathbb D}, \quad \text{if $\|A\|_2 \le 1$ and $\|B\|_2 \le 1$.}
\end{equation}
Let us note that such a result also holds for functions in more than two variables, in the sense defined in Section~\ref{sec:multivariate}, because the involved matrices are doubly commuting~\cite[Sec. 1.5.9 (g)]{Nikolski2002}.

\section{Norm estimates for matrix-valued mappings}

Our proof of~\eqref{eq:newresult} is based on the matrix-valued version of~\eqref{eq:scalarcrouzeixpalencia}, which is equivalent to stating that $W(A)$ is not only a $(1+\sqrt{2})$-spectral set but in fact a \emph{complete} $(1+\sqrt{2})$-spectral set for $A$. The existence of such a matrix-valued version is stated without proof in~\cite{CrouzeixPalencia2017}. Given the centrality of this result in our derivation, we feel it worthwhile to include a detailed proof.

Consider a smooth, bounded, convex domain $\Omega \subset \mathbb C$ and a matrix-valued function $F: \Omega \to \C^{m\times p}$ that is (element-wise) analytic in $\Omega$ and admits a continuous extension to $\overline{\Omega}$.
We will work with the maximum of the matrix $2$-norm:
\[
 \|F\|_\Omega := \sup_{z \in \Omega} \| F(z) \|_2.
\]
The function defined by
\[
 G(z) = \frac{1}{2\pi \mathrm{i}} \oint_{\partial \Omega} F^*(\sigma) \frac{\mathrm{d}\sigma}{\sigma-z}
\]
is clearly analytic in $\Omega$ and, by applying~\cite[Lemma 2.1]{CrouzeixPalencia2017} to each entry of $G$, it also admits a continuous extension to $\overline{\Omega}$.
\begin{lemma} \label{lemma:1}
 For the function $G$ defined above, it holds that $\|G\|_\Omega \le \|F\|_\Omega$.
\end{lemma}
\begin{proof}
Given arbitrary vectors $u\in \C^m$, $v\in \C^p$ with $\|u\|_2 = \|v\|_2 = 1$, we can apply~\cite[Lemma 2.1]{CrouzeixPalencia2017} to the scalar functions \[
u^* G(z) v = \frac{1}{2\pi \mathrm{i}} \oint_{\partial \Omega} \overline{v^* F(\sigma) u} \frac{\,\mathrm{d}\sigma}{\sigma-z}
\]
and $v^* F(z) u$ to obtain $|u^* G(z) v| \le \| v^* F(\cdot) u \|_\Omega \le \|F\|_\Omega$. By taking the supremum over all such $u,v$ we obtain
$\|G(z)\|_2 \le \|F\|_\Omega$.
\end{proof}

For a matrix $A \in \C^{n\times n}$ with eigenvalues contained in $\Omega$, we define $F(A)$ by replacing the function $f_{ij}(z)$ at each entry $(i,j)$ of $F(z)$ by the corresponding matrix function $f_{ij}(A)$:
\[
 F(A) = \begin{bmatrix}
f_{11}(A) & \cdots & f_{1p}(A) \\
\vdots & & \vdots \\
f_{m1}(A) & \cdots & f_{mp}(A)
\end{bmatrix}.
\]
Using the Cauchy integral formula and the Kronecker product $\otimes$, we can write
\[
 F(A) = \frac{1}{2\pi \mathrm{i}} \oint_{\partial \Omega} F(\sigma) \otimes (\sigma I -A)^{-1} \,\mathrm{d}\sigma.
\]
and, in turn,
\[
 G(A) = \frac{1}{2\pi \mathrm{i}} \oint_{\partial \Omega} F(\sigma)^* \otimes (\sigma I -A)^{-1} \,\mathrm{d}\sigma.
\]
\begin{lemma} \label{lemma:2} Assume that $W(A)\subset \Omega$. Then the matrices $F(A)$, $G(A)$ defined above satisfy $\|F(A)+G(A)^*\|_2 \le 2 \|F\|_\Omega$.
\end{lemma}
\begin{proof}
Without loss of generality, we may assume $\|F\|_\Omega = 1$.
Set $S:=F(A)+G(A)^*$. Using a counterclockwise oriented arclength parametrization $\sigma = \sigma(s)$ of $\Omega$, we obtain 
\begin{align*}
 S & = \frac{1}{2\pi \mathrm{i}} \oint_{\partial \Omega} F(\sigma) \otimes \big[ (\sigma I -A)^{-1} \,\mathrm{d}\sigma + (\bar \sigma I -A^*)^{-1} \,\mathrm{d}\bar \sigma \big]  \\
 &= 
 \oint_{\partial \Omega} F(\sigma) \otimes \mu(\sigma,A)\,\mathrm{d}s,
\end{align*}
where $\mu(\sigma,A) = \frac{1}{2\pi}\big(\nu (\sigma I - A)^{-1} + \overline \nu (\bar \sigma I -A^*)^{-1}\big)$ and $\nu$ denotes the unit outward normal vector of $\partial \Omega$ at $\sigma$. As discussed in~\cite[Sec. 2]{CrouzeixPalencia2017}, the matrix $\mu(\sigma,A)$ is Hermitian positive definite and satisfies
\[
 \oint_{\partial \Omega} \mu(\sigma,A)\,\mathrm{d}s = 2I.
\]
For matrices $X \in \C^{n\times p}$, $Y \in \C^{n\times m}$, we set $x = \text{vec}(X)$, $y = \text{vec}(Y)$, let $\langle\cdot,\cdot\rangle_F$ denote the Frobenius inner product of matrices and derive
 \begin{eqnarray*}
 | \langle Sx,y\rangle | &=&  \Big| \oint_{\partial \Omega} \big\langle \mu(\sigma,A)XF(\sigma)^T,Y \big\rangle_F \,\mathrm{d}s \Big| \le 
 \oint_{\partial \Omega} \big| \big\langle \mu(\sigma,A)XF(\sigma)^T,Y \big\rangle_F \big| \,\mathrm{d}s \\
 &\le& \oint_{\partial \Omega} \big\langle \mu(\sigma,A)X F(\sigma)^T,XF(\sigma)^T \big\rangle_F^{1/2}\, \big\langle \mu(\sigma,A)Y,Y \big\rangle_F^{1/2}\,\mathrm{d}s
 \end{eqnarray*}
where we used the Cauchy-Schwarz inequality in the inner product $\langle \mu(\sigma,A) \cdot,\cdot\rangle_F$. Combined with
\begin{align*}
\langle \mu(\sigma,A)XF(\sigma)^T,XF(\sigma)^T \big\rangle_F &=  
 \text{trace}\big( \mu(\sigma,A) X F(\sigma)^T \overline{F(\sigma)} X^* \big) \\ &= 
 \text{trace}\big( X^* \mu(\sigma,A) X \overline{F(\sigma)^* F(\sigma)} \big) \\
 &\le \lambda_{\max}( F(\sigma)^* F(\sigma) )\,\text{trace}\big( X^* \mu(\sigma,A) X  \big) \\ &\le \text{trace}\big( X^* \mu(\sigma,A) X  \big),
\end{align*}
we obtain
 \begin{eqnarray*}
 | \langle Sx,y\rangle | &\le&  
 \oint_{\partial \Omega} \big\langle \mu(\sigma,A)X,X \big\rangle_F^{1/2}\, \big\langle \mu(\sigma,A)Y,Y \big\rangle_F^{1/2}\,\mathrm{d}s \\
 &\le & \Big( \oint_{\partial \Omega} \big\langle \mu(\sigma,A)X,X \big\rangle_F\,\mathrm{d}s \Big)^{1/2}
 \Big( \oint_{\partial \Omega} \big\langle \mu(\sigma,A)Y,Y \big\rangle_F\,\mathrm{d}s \Big)^{1/2} \\
 &=& \Big( \Big\langle \oint_{\partial \Omega}  \mu(\sigma,A)\,\mathrm{d}s\, X,X \Big\rangle_F \Big)^{1/2}
 \Big( \Big\langle \oint_{\partial \Omega}  \mu(\sigma,A)\,\mathrm{d}s\, Y,Y \Big\rangle_F\,\mathrm{d}s \Big)^{1/2} \\
 &=& 2 \|X\|_F \|Y\|_F = 2\|x\|_2\|y\|_2. 
\end{eqnarray*}
This proves $\|S\|_2\le 2$.
\end{proof}

\begin{theorem} \label{theorem:crouzeixpalencia}
Assume that $W(A)\subset \Omega$. Then 
\[
\|F(A)\|_2 \le (1+\sqrt{2}) \|F\|_\Omega.
\]
\end{theorem}
\begin{proof}
The result follows from Lemma~\ref{lemma:1} and Lemma~\ref{lemma:2} in a manner entirely analogous to the derivation in~\cite[Sec. 2]{Ransford2018}. 
\end{proof}

\section{Norm estimate for bivariate matrix functions}

We now extend Theorem~\ref{theorem:crouzeixpalencia} to the bivariate case. For smooth, bounded, and convex domains $\Omega_A,\Omega_B \subset \C$ we consider a matrix valued function $F:\Omega_A\times\Omega_B\to \C^{m\times p}$ that is analytic in $\Omega_A\times \Omega_B$ and continuous on $\overline{\Omega}_A \times \overline{\Omega}_B$. For matrices $A \in\C^{n_A\times n_A}$ and $B \in \C^{n_B\times n_B}$ with the numerical ranges contained in $\Omega_A$ and $\Omega_B$, respectively, we define
\begin{equation} \label{eq:bivariatematrixvalued}
 F\{A,B\} := -\frac{1}{4\pi^2} \oint_{\partial \Omega_A} \oint_{\partial \Omega_B} F(x,y) \otimes (yI-B)^{-1} \otimes (xI-A)^{-1}\,\text{d}y\,\text{d}x.
\end{equation}
This includes~\eqref{def:bivariatefunction} as a special case for $m = p = 1$.
%In the following theorem, we define $\|F\|_{\Omega_A\times\Omega_B}:=\sup\{\|F(x,y)\|_2:\, x\in \Omega_A, y\in \Omega_B\}$.

\begin{theorem}  \label{theorem:main}
For the matrix $F\{A,B\}$ defined above it holds that
\[
 \|F\{A,B\}\|_2 \le (1+\sqrt{2})^2 \|F\|_{W(A)\times W(B)},
\]
\end{theorem}
\begin{proof}
For $x \in \overline{\Omega}_A$, we let $F_x(y):=F(x,y)$. Inserting
\[
 F_x(B) = \frac{1}{2\pi \mathrm{i}} \oint_{\partial \Omega_B} F(x,y) \otimes (yI-B)^{-1}  \,\text{d}y
\]
into~\eqref{eq:bivariatematrixvalued} yields
\[
 F\{A,B\} = \frac{1}{2\pi \mathrm{i}} \oint_{\partial \Omega_A} F_x(B) \otimes (xI-A)^{-1}\,\text{d}x.
\]
This allows us to view $F\{A,B\}$ as the evaluation of the matrix valued function
$F_B(x):=F_x(B)$ in $A$, that is, $F\{A,B\}=F_B(A)$. Using Theorem~\ref{theorem:crouzeixpalencia} twice gives
\begin{align*}
 \|f\{A,B\}\|_2 & \le (1+\sqrt{2}) \sup_{x\in \Omega_A} \|F_B(x)\|_2  = (1+\sqrt{2}) \sup_{x\in \Omega_A} \|f_x(B)\|_2 \\ & \le (1+\sqrt{2})^2 \sup_{x\in \Omega_A \atop y\in \Omega_B} \|f(x,y)\|_2.
\end{align*}
As this inequality holds for any $\Omega_A,\Omega_B$  containing $W(A),W(B)$, the statement of the theorem follows by continuity.
\end{proof}

\section{Extension to multivariate functions} \label{sec:multivariate}

The result of Theorem~\ref{theorem:main} extends without difficulty to functions in more than two variables. Let 
$\Omega_i \subset \C$ be smooth, bounded, and convex domains for $i = 1,\ldots,d$ and consider a function $F(x_1,\ldots,x_d) \in \C^{m\times p}$ that is analytic in  $\Omega_1\times\cdots \times \Omega_d$. For matrices $A_i \in\C^{n_i \times n_i}$ with eigenvalues contained in $\Omega_i$, we define recursively
\[
 F\{A_1,\ldots,A_d\} = \frac{1}{2\pi \mathrm{i}} \oint_{\partial \Omega_1} F_{x_1}\{A_2,\ldots,A_d\} \otimes (x_1 I-A)^{-1}\,\text{d}x_1,
\]
where $F_{x_1}(x_2,\ldots,x_d):=F(x_1,x_2,\ldots,x_d)$. Applying the technique from the proof of Theorem~\ref{theorem:main} recursively, we obtain
\begin{equation} \label{eq:multid}
  \|F\{A_1,\ldots,A_d\}\|_2 \le (1+\sqrt{2})^d \|F\|_{W(A_1) \times \cdots \times W(A_d)}. %\sup\big\{ \|F(x_1,\ldots,x_d)\|_2:\, x_1 \in W(A_1),\, \ldots,\, x_d\in W(A_d)\big\},
\end{equation}
provided that $W(A_1)\subset \Omega_1, \ldots, W(A_d)\subset \Omega_d$.

A straightforward extension of the example from the introduction shows that the constant in~\eqref{eq:multid} must be at least $2^d$; hence, the exponential growth with respect to $d$ is unavoidable. On the other hand, the constant can be decreased to $(1+\sqrt{2})^{d-k}$ if one assumes that there are $k$ normal matrices among $A_1,\ldots,A_d$.

\section{Application to derivatives of matrix functions} \label{sec:derivatives}

The Fr\'echet derivative of the matrix function $f(A)$ is the linear operator 
$Df\{A\}: \C^{n \times n} \to \C^{n \times n}$ satisfying
\[
 f(A+\Delta) = f(A) + Df\{A\}(\Delta) + O(\|\Delta\|_2^2)
\]
for all $\Delta \in \C^{n\times n}$ of sufficiently small norm. The norm $\|Df\{A\}\|_2$ induced by the Frobenius norm is the (absolute) condition number of $f(A)$, an important quantity to assess the effect of perturbations (e.g., due to roundoff error) on matrix functions~\cite[Chap. 3]{Higham2008}.
Most existing bounds for $f(A)$ proceed via diagonalizing $A$ and inevitably involve the squared condition number of the eigenvector matrix; see, e.g.,~\cite[Theorem 3.15]{Higham2008}. A notable exception is Corollary 3.2 in~\cite{Deadman2016}, which derives a bound in terms of the pseudospectrum of $A$.
The following corollary presents a bound in terms of the maximum of the derivative on the numerical range.

\begin{corollary}
Let $A\in \C^{n\times n}$ and consider an analytic function $f: \Omega \to \C$ with $W(A) \subset \Omega$. Then
 \[
  \|Df\{A\}\|_2 \le (1+\sqrt{2})^2\, \|f^\prime\|_{W(A)}.
 \]
\end{corollary}
\begin{proof}
The divided difference
\[
f^{[1]}(x,y):=f[x,y] = \left\{
\begin{array}{ll}
 \frac{f(x)-f(y)}{x-y}, & \text{for } x\not=y, \\
 f^\prime(x), & \text{for } x=y, \\
\end{array}\right.
\]
is analytic in $\Omega \times \Omega$. Moreover, as explained in~\cite{Kressner2014bivariate}, the corresponding bivariate matrix function $f^{[1]}\{A,A^T\}$ is the (canonical) matrix representation of the linear operator $Df\{A\}$. In turn, Theorem~\ref{theorem:main} gives
\[
 \|Df\{A\}\|_2 = \|f^{[1]}\{A,A^T\}\|_2 \le (1+\sqrt{2})^2 \|f^{[1]}\|_{W(A)\times W(A)}.
\]
To simplify the last term, we note that for arbitrary $x,y\in W(A)$, the line segment $\gamma$  from $x$ to $y$ is contained in $W(A)$ and thus
\[
 |f(y)-f(x)| = \Big| \int_\gamma f^\prime(z) dz \Big| \le |y-x| \max_{z\in\gamma}|f^\prime(z)|.
\]
Hence, $|f^{[1]}(x,y)| \le \sup_{z\in W(A)} |f^\prime(z)|$ for all $x,y\in W(A)$, which completes the proof.
\end{proof}

\section{Application to convergence analysis of Krylov subspace methods} \label{sec:krylov}

As nicely explained in~\cite{BadeaBeckermann2013}, norm bounds of the form~\eqref{eq:scalarcrouzeixpalencia} are an essential ingredient in deriving error bounds for approximations of matrix functions. Theorem~\ref{theorem:main} can serve the same purpose for bivariate matrix functions. To illustrate this, we consider the Arnoldi method from~\cite{Kressner2019bivariate} for approximating the matrix-vector product $f\{A,B\}c$, where $c$ is the vectorization of a rank-one matrix: $c = \text{vec}(c_A c_B^T)$ with $c_A \in \C^m$, $c_B \in \C^n$.

The method from~\cite{Kressner2019bivariate} applies $k$ steps of the standard Arnoldi process to generate an orthonormal basis $U_k$ of the 
$k$-dimensional Krylov subspaces $\calK_k(A,c_A)$ and, similarly, an 
orthonormal bases $V_\ell$ of 
$\calK_\ell(B,c_B)$. It returns the approximation 
\begin{equation} \label{eq:arnoldiapprox}
 x_{k,\ell} = ( V_\ell \otimes U_k) y_{k,\ell},
\end{equation}
which involves the reduced bivariate matrix function
\[
 y_{k,\ell} = f\big\{U_k^\ast A U_k ,V_\ell^\ast B V_\ell \big\}
 ( V_\ell \otimes U_k)^\ast c.
\]
As explained in~\cite{Kressner2019bivariate}, this general framework unifies existing Krylov subspace methods for various types of matrix equations and the Fr\'echet derivative. Theorem~\ref{theorem:main} allows us to link the approximation error for $y_{k,\ell}$ to a (bivariate) polynomial approximation problem.
\begin{corollary} \label{theorem:convergence}
Consider an analytic function $f:\Omega_A\times\Omega_B\to \C$
with $W(A) \subset \Omega_A$ and $W(B) \subset \Omega_B$. Then the approximation~\eqref{eq:arnoldiapprox} returned by the Arnoldi method satisfies
the error bound
\[
 \|f\{A,B\} c  - x_{k,\ell} \|_2 \le 2 (1+\sqrt{2})^2 \|c\|_2 \inf_{p \in \Pi_{k-1,\ell-1}} \| f-p \|_{W(A)\times W(B)},
\]
where $\Pi_{k,\ell}$ denote the set of all bivariate polynomials of degree at most $(k,\ell)$.
\end{corollary}
\begin{proof}
The proof of Theorem~4.3 in~\cite{Kressner2019bivariate} implies that
\[
 \|f\{A,B\} c  - x_{k,\ell} \|_F \le \big( \| e\{A,B \} \|_2 + \| e\{ U_k^\ast A U_k ,V_\ell^\ast B V_\ell \}\|_2  \big) \|c\|_2
 \]
 with $e = f-p$ for arbitrary $p \in \Pi_{k,\ell}$.
 Applying Theorem~\ref{theorem:main} to 
 $\| e\{A,B \} \|_2$, $\| e\{ U_k^\ast A U_k ,V_\ell^\ast B V_\ell \}\|_2$ and noting that
 $W(U_k^\ast A U_k) \subset W(A)$, $W(V_\ell^\ast B V_\ell) \subset W(B)$ concludes the proof.
\end{proof}

\bibliographystyle{siamplain}
%\bibliography{matrixfunctions}

\bibliography{anchp}

\begin{thebibliography}{10}

\bibitem{Ando1963}
{\sc T.~And\^{o}}, {\em On a pair of commutative contractions}, Acta Sci. Math.
  (Szeged), 24 (1963), pp.~88--90.

\bibitem{BadeaBeckermann2013}
{\sc C.~Badea and B.~Beckermann}, {\em Spectral sets}, in Handbook of Linear
  Algebra, Chapman and Hall/CRC, Boca Raton, FL, 2nd~ed., 2013.
\newblock Chapter 37.

\bibitem{Biller2007}
{\sc H.~Biller}, {\em Analyticity and naturality of the multi-variable
  functional calculus}, Expo. Math., 25 (2007), pp.~131--163,
  \url{https://doi.org/10.1016/j.exmath.2006.09.001}.

\bibitem{Cadwell2018}
{\sc T.~Caldwell, A.~Greenbaum, and K.~Li}, {\em Some extensions of the
  {C}rouzeix-{P}alencia result}, SIAM J. Matrix Anal. Appl., 39 (2018),
  pp.~769--780, \url{https://doi.org/10.1137/17M1140832}.

\bibitem{CrouzeixPalencia2017}
{\sc M.~Crouzeix and C.~Palencia}, {\em The numerical range is a
  {$(1+\sqrt{2})$}-spectral set}, SIAM J. Matrix Anal. Appl., 38 (2017),
  pp.~649--655, \url{https://doi.org/10.1137/17M1116672}.

\bibitem{Deadman2016}
{\sc E.~Deadman and S.~D. Relton}, {\em Taylor's theorem for matrix functions
  with applications to condition number estimation}, Linear Algebra Appl., 504
  (2016), pp.~354--371, \url{https://doi.org/10.1016/j.laa.2016.04.010}.

\bibitem{Gil2011}
{\sc M.~Gil'}, {\em Norm estimates for functions of two non-commuting
  matrices}, Electron. J. Linear Algebra, 22 (2011), pp.~504--512.

\bibitem{Glader2018}
{\sc C.~Glader, M.~Kurula, and M.~Lindstr\"{o}m}, {\em Crouzeix's conjecture
  holds for tridiagonal {$3\times 3$} matrices with elliptic numerical range
  centered at an eigenvalue}, SIAM J. Matrix Anal. Appl., 39 (2018),
  pp.~346--364, \url{https://doi.org/10.1137/17M1110663}.

\bibitem{Greenbaum2018}
{\sc A.~Greenbaum and M.~L. Overton}, {\em Numerical investigation of
  {C}rouzeix's conjecture}, Linear Algebra Appl., 542 (2018), pp.~225--245,
  \url{https://doi.org/10.1016/j.laa.2017.04.035}.

\bibitem{Higham2008}
{\sc N.~J. Higham}, {\em Functions of matrices}, SIAM, Philadelphia, PA, 2008.

\bibitem{Kressner2014bivariate}
{\sc D.~Kressner}, {\em Bivariate matrix functions}, Oper. Matrices, 8 (2014),
  pp.~449--466, \url{https://doi.org/10.7153/oam-08-23}.

\bibitem{Kressner2019bivariate}
{\sc D.~Kressner}, {\em A {K}rylov subspace method for the approximation of
  bivariate matrix functions}, in Structured matrices in numerical linear
  algebra, vol.~30 of Springer INdAM Ser., Springer, Cham, 2019, pp.~197--214.

\bibitem{Nikolski2002}
{\sc N.~K. Nikolski}, {\em Operators, functions, and systems: an easy reading.
  {V}ol. 2}, vol.~93 of Mathematical Surveys and Monographs, American
  Mathematical Society, Providence, RI, 2002.

\bibitem{Ransford2018}
{\sc T.~Ransford and F.~L. Schwenninger}, {\em Remarks on the
  {C}rouzeix-{P}alencia proof that the numerical range is a
  {$(1+\sqrt2)$}-spectral set}, SIAM J. Matrix Anal. Appl., 39 (2018),
  pp.~342--345, \url{https://doi.org/10.1137/17M1143757}.

\bibitem{Skripka2020}
{\sc A.~Skripka}, {\em Untangling noncommutativity with operator integrals},
  Notices Amer. Math. Soc., 67 (2020), pp.~45--55.

\bibitem{Skripka2019}
{\sc A.~Skripka and A.~Tomskova}, {\em Multilinear operator integrals},
  vol.~2250 of Lecture Notes in Mathematics, Springer, Cham, 2019,
  \url{https://doi.org/10.1007/978-3-030-32406-3}.

\bibitem{Starke1993}
{\sc G.~Starke}, {\em Fields of values and the {ADI} method for nonnormal
  matrices}, Linear Algebra Appl., 180 (1993), pp.~199--218,
  \url{https://doi.org/10.1016/0024-3795(93)90531-R}.

\end{thebibliography}

\end{document}